\documentclass[10pt,reqno,final]{amsart}
\usepackage{amsfonts} 
\usepackage{amsthm}
\usepackage{amsmath}
\usepackage{amssymb}
\usepackage{mathrsfs}
\usepackage[latin1]{inputenc}
\usepackage{color}
\usepackage{graphicx}
\usepackage{enumerate}
\usepackage{bbm}
\usepackage{changes} %final
\usepackage[margin=2cm]{geometry}
\newcommand*{\C}{\mathbb{C}}

\newcommand*{\N}{\mathbb{N}}

\newcommand*{\R}{\mathbb{R}}

\numberwithin{equation}{section}

\theoremstyle{plain}      \newtheorem{theorem}{Theorem}[section]
\newtheorem{corollary}[theorem]{Corollary}
\newtheorem{proposition}[theorem]{Proposition}

\theoremstyle{remark}     \newtheorem{remark}[theorem]{Remark}
\newtheorem{lemma}[theorem]{Lemma}

\theoremstyle{definition} \newtheorem{definition}[theorem]{Definition}

\bibliography{bibfile}
\begin{document}
\title[Analysis of Stochastic Quantization]{Analysis of Stochastic Quantization for the fractional Edwards Measure}
\author[Wolfgang Bock]{Wolfgang Bock}
\address{Technomathematics Group, University of Kaiserslautern}
\email{bock@mathematik.uni-kl.de}
%\urladdr{http://www.math.univ.edu/$\sim$johndoe \bf(optional)}

\author[Torben Fattler]{Torben Fattler}
%\thanks{}
\address{Functional Analysis and Stochastic Analysis Group, University of Kaiserslautern}
\email{fattler@mathematik.uni-kl.de}

%{\newline%
%\indent }%
%\email[]{}%
%\urladdr{}
%\thanks{}
\thanks{}
%\thanks{This paper is in final form and no version of it will be submitted for
%publication elsewhere.}
\date{\today}
\subjclass{} %
\keywords {}%
%
%\dedicatory{}
%
\maketitle
\begin{abstract}
	In \cite{BFS16} the existence of a diffusion process whose invariant measure is the fractional polymer or Edwards measure for fractional Brownian motion in dimension $d\in\mathbb{N}$ with Hurst parameter $H\in(0,1)$ fulfilling $dH < 1$ is shown. The diffusion is constructed via Dirichlet form techniques in infinite dimensional (Gaussian) analysis. By providing a Fukushima decomposition for the stochastic quantization of the fractional Edwards measure we prove 
	that the constructed process solves weakly a stochastic differential equation in infinite dimension for quasi-all starting points.
	Moreover, the solution process is driven by an Ornstein--Uhlenbeck process taking values in an infinite dimensional distribution space and is unique, in the sense that the underlying Dirichlet form is Markov unique. 
	The equilibrium measure, which is by construction the fractional Edwards measure, is specified to be an extremal Gibbs state and therefore, the constructed stochastic dynamics is time ergodic. The studied stochastic differential equation provides in the language of polymer physics the dynamics of the bonds, i.e.~stochastically spoken the noise of the process. An integration leads then to polymer paths. We show that if one starts with a continuous polymer configuration the integrated process stays almost surely continuous during the time evolution.
\end{abstract}
\section{Introduction}

	For a given probability measure $\nu$ on a measurable space $(X,\mathcal{F})$ the stochastic quantization of $\nu$ means the construction of a Markov process which has $\nu$ as an invariant measure. Stochastic quantization has been studied first by Parisi and Wu for applications in quantum field theory, which were extended to Euclidean quantum fields, see \cite{PW81}. The Markov process, which is obtained, is parametrized w.r.t.~to a new time parameter, which is often denoted as \lq compute time\rq. This notion is due to the fact, that one can use the stochastic quantization in order to construct a numerical scheme to sample path with a given probability distribution, \cite{Guha}.
	
	The two-dimensional polymer measure or Edward's measure is informally given as
	$$
	d\mu_g = Z^{-1} e^{-gL} d\mu_0,$$
	where $\mu_0$ denotes the Wiener measure, $L$ the self-intersection local time of Brownian motion and $Z$ is a normalization constant.  The self-intersection local time can be interpreted as the time the process spends on its trajectory - or in particular - it counts the self-crossings the process undertakes. The Edwards measure thus penalizes every self-intersection exponentially. Note however that it gives for $g<\infty$ only weakly self-avoiding paths, i.e.~the path is allowed to cross itself but with a growing number of self-intersections the next self-intersection is becoming more and more unlikely.  
	
	Albeverio, Röckner, Hu and Zhou use Dirichlet form methods to construct a Markov process associated to the Edwards measure in two dimensions \cite{ARHZ}.   
	
	Also intersection local times $L$ of Brownian motion have been studied for a long
	time and by many authors, see e.g. \cite{ARHZ}, \cite{bass}, \cite{fcs}, \cite{dvor2}, \cite{he}, \cite{imke}, \cite{legall}, \cite{lyons} and \cite{sym}-\cite{yor2}, 
	the intersections of Brownian motion paths have been studied even since the
	Forties, see e.g.~\cite{levy}. One can consider intersections of sample paths with
	themselves or e.g. with other, independent Brownian motions e.g.~\cite{wolp}, one
	can study simple \cite{dvor2} or $n$-fold intersections e.g. \cite{dvor3}, \cite{lyons} and one
	can ask all of these questions for linear, planar, spatial or - in general -
	$d$-dimensional Brownian motion: self-intersections become
	increasingly scarce as the dimension $d$ increases. A well-written monograph about self-avoiding random walks is provided by N.~Madras and G.~Slade \cite{madras}. 
	
	A somewhat informal but very suggestive definition of self-intersection
	local time of 
	a Gaussian process $Y$
	is in terms of an integral over Dirac's -
	or Donsker's - $\delta $-function
	\[
	L(Y)\equiv \int d^2t\,\delta (Y(t_2)-Y(t_1)), 
	\]
	where for now $Y=B$ is a Brownian motion, intended to sum up the contributions from each pair of ''times'' $t_1,t_2$
	for which the process $Y$ is at the same point, see e.g.~\cite{dFHWS97}. In Edwards' modeling
	of long polymer molecules by Brownian motion paths, $L$ is used to model the
	''excluded volume'' effect: different parts of the molecule should not be
	located at the same point in space. As another application, Symanzik \cite
	{sym} introduced $L$ as a tool in constructive quantum field theory.
	
	A rigorous definition, such as e.g.~through a sequence of Gaussians
	approximating the $\delta $-function, will lead to increasingly singular
	objects and will necessitate various ''renormalizations'' as the dimension $d$
	increases. For $d>1$ the expectation will diverge in the limit and must be
	subtracted, see e.g.~\cite{legall}, \cite{varadhan}, as a side effect such a local time will then no
	more be positive. For $d>3$ various further renormalizations have been
	proposed in \cite{watanabe} that will make $L$ into a well-defined generalized
	function of Brownian motion. For $d=3$ a multiplicative renormalization
	gives rise to an independent Brownian motion as the weak limit of
	regularized and subtracted approximations to $L$, see \cite{yor2}; another
	renormalization has been constructed by Westwater to make the Gibbs factor $%
	\exp(-gL)$ of the polymer model well-defined, see \cite{west}.
	
	In this article we first introduce the setting along the lines of White Noise or Gaussian Analysis, using  the fractional White Noise measure. This can be compared to the approach in \cite{OH03}. Moreover the results from \cite{Po97} concerning the gradient are extended to this setting. 
	In \cite{BFS16} the stochastic quantization for the fractional Edwards measure $\nu_{\scriptscriptstyle{g}}$ is studied using the framework of Dirichlet forms. The existence of a Markov process which has $\nu_{\scriptscriptstyle{g}}$ as invariant measure  is based on the results of \cite{Hu2001} and \cite{HNS06}, which show that the self-intersection local time in the case $Hd<1$ is Meyer-Watanabe differentiable. The closability of the gradient Dirichlet form is then shown by an integration by parts argument. The irreducibility follows as in the Brownian case, see \cite{ARHZ}. 
	
	In this article we carryout a further analysis of the underlying objects. Using the theory of Dirichlet forms we show that the stochastic quantization of the fractional Edwards measure solves weakly a stochastic differential equation (SDE) in infinite dimension given by
	\begin{align*}
	dX_t = \sqrt{2}\,dW_t + b(X_t)\,dt,\quad X_0=\omega,
	\end{align*}
	for quasi-all starting points $\omega$ in an infinite dimensional state space $\mathcal{N}'$. Here the drift term $b$ is determined by the gradient of the self-intersection local time of fractional Brownian motion with Hurst parameter $H$ fulfilling $Hd<1$, where $d\in\mathbb{N}$ denotes space dimension. The solution process is driven by a Brownian motion having an intrinsic linear drift. Such a process is due to \cite{Kuo} characterized as an \emph{Ornstein--Uhlenbeck process} taking values in an infinite dimensional space. The solution process is unique, in the sense that the underlying Dirichlet form is Markov unique. The equilibrium measure, which is by construction the fractional Edwards measure, is specified to be the \emph{extremal Gibbs state}. The SDE under consideration provides in the language of polymer physics the dynamics of the bonds, i.e.~stochastically spoken the noise of the process. An integration leads then to polymer paths. We show that if one starts with a continuous polymer configuration during the evolution the integrated process stays almost surely continuous.

%\section{Introduction}
%In \cite{BFS16} the stochastic quantization for the fractional Edwards measure is studied. Here we carryout a further analysis of the underlying objects.
\section{Framework}
For $d\in\mathbb{N}$ and \emph{Hurst parameter} $H\in(0,1)$ a \emph{fractional Brownian motion in dimension $d$} is a $\mathbb{R}^d$-valued centered Gaussian process $\big(B^{\scriptscriptstyle{H}}_t\big)_{t\ge 0}$
with covariance
\begin{align*}
\text{cov}_{\scriptscriptstyle{H}}(t,s):=\mathbb{E}\big[B^{\scriptscriptstyle{H}}_t B^{\scriptscriptstyle{H}}_s\big]=\frac{1}{2}\left(
t^{2H}+s^{2H}-|t-s|^{2H}\right),\quad s,t\in [0,\infty),
\end{align*}
defined on a probability space $(\Omega,\mathcal{F}, P)$. Here $\mathbb{E}$ denotes the mathematical expectation with respect to the probability measure $P$.
For $s\in(0,\infty)$ let $\Theta_s:=\mathbbm{1}_{[0,s)}$ and set $\big(\Theta_{s},\Theta_{t}\big)_{\scriptscriptstyle{H}}:=\text{cov}_{\scriptscriptstyle{H}}(t,s)$ for $s,t\in [0,\infty)$. Moreover, let $X:=\text{span}\big\{\Theta_s\,\big|\,s>0\big\}$. Hence $x,y\in X$ are simple functions of the form
\begin{align*}
x=\sum_{i=1}^n\alpha_i\Theta_{s_{i}},\quad y=\sum_{j=1}^m\beta_j\Theta_{t_{j}},
\end{align*}
with $n,m\in\mathbb{N}$ and
\begin{align*}
\big(x,y\big)_{\scriptscriptstyle{H}}:=\sum_{i=1}^n\sum_{j=1}^m\alpha_i\,\beta_j\big(\Theta_{s_i},\Theta_{t_j}\big)_{\scriptscriptstyle{H}},
\end{align*}
defines an inner product on $X$. Taking the abstract completion of the inner product space $\big(X,(\cdot,\cdot)_{\scriptscriptstyle{H}}\big)$ we obtain a Hilbert space $\big(\mathcal{H},(\cdot,\cdot)_{\scriptscriptstyle{H}}\big)$, where the scalar product extending $(\cdot,\cdot)_{\scriptscriptstyle{H}}$ to $\mathcal{H}$ is denoted by the same symbol.

Moreover, $\big(\mathcal{H},(\cdot,\cdot)_{\scriptscriptstyle{H}}\big)$ has a \emph{countable orthonormal basis} $\beta=\big(\eta_k\big)_{k\in\mathbb{N}}$. For $k\in\mathbb{N}$ let $\lambda_k\in\mathbb{R}$ such that 
\begin{align*}
1<\lambda_1<\lambda_2<\ldots<\lambda_k<\lambda_{k+1}<\ldots\quad\text{and}\quad\sum_{k=1}^\infty\frac{1}{\lambda_k^2}<\infty.
\end{align*}
Next we consider
\begin{align*}
\mathcal{H}\ni f\mapsto Af:=\sum_{k=1}^\infty\lambda_k\,\big( f,\eta_k\big)_{\scriptscriptstyle{H}}\,\eta_k\in\mathcal{H}
\end{align*}
and define for $p\in\mathbb{N}$
\begin{align*}
\mathcal{H}_p:=\big\{f\in\mathcal{H}\,\big|\,\Vert A^pf\Vert_{\scriptscriptstyle{H}}<\infty\big\}\quad\text{and}\quad \mathcal{N}:=\bigcap_{p\in\mathbb{N}}\mathcal{H}_p,
\end{align*}
where $\Vert\cdot\Vert_{\scriptscriptstyle{H}}$ denotes the induced norm on $\mathcal{H}$.    
Then $\mathcal{N}$ is a countably Hilbert space, which is Fr\'{e}chet and nuclear, compare e.g.~\cite{Ob94}. Its topological dual is given by 
$$
\mathcal{N}':=\bigcup_{p\in\mathbb{N}}\mathcal{H}_{-p},
$$
for an analogous construction see e.g.~\cite{HKPS93}.
Thus we obtain the Gel'fand triple
$$
\mathcal{N}\subset \mathcal{H} \subset\mathcal{N}'.
$$
In what follows we denote complexifications by a subscript $\C$.

Now by the Bochner-Minlos-Sazonov theorem, see e.g.~\cite{BK95} or \cite{hida70}, we define a Gaussian measure $\mu_{\scriptscriptstyle{H}}$ on $ \mathcal{N}'$ by 
$$\int_{\mathcal{N}'} \exp\left(i \langle \omega, \xi \rangle_{\scriptscriptstyle{H}} \right)\, d\mu_{\scriptscriptstyle{H}}(\omega) := \exp\left(-\frac{1}{2} \| \xi \|_{\scriptscriptstyle{H}} \right).$$

\begin{remark}\label{rem_fullsupp}
Note that the measure has full support, i.e.~every open set has positive measure. This can be seen by \cite[Theorem 6]{KSW} or the 
fact that the measure is quasi translation invariant w.r.t.~shifts in direction of the subspace $\mathcal{N}$ which is dense in $\mathcal{N}'$, compare e.g.~\cite[Chapter 4B]{HKPS93}.

\end{remark}
We obtain the probability space $(\mathcal{N}', \mathcal{C}_{\sigma}, \mu_{\scriptscriptstyle{H}} )$.
Here $\mathcal{C}_{\sigma}(\mathcal{N}') :=\sigma(\mathcal{C}^{\xi_1, \dots , \xi_n }_{F_1, \dots , F_n})$ denotes the $\sigma$-algebra of cylinder sets
\begin{align}
\mathcal{C}^{\xi_1, \dots , \xi_n }_{F_1, \dots , F_n} 
= \Big\{ \omega \in \mathcal{N'} \, \big| \langle \xi_1,\omega \rangle_{\scriptscriptstyle{H}} \in F_1, \dots ,\langle \xi_n, \omega \rangle_{\scriptscriptstyle{H}} \in F_n,~
 \xi_i \in \mathcal{N},~F_j \in \mathcal{B}(\R), j=1,\dots ,n, \,\, n \in \N\Big\},
\end{align}
where $\mathcal{B}(\R)$ denotes the $\sigma$-algebra of Borel sets in $\R$. 

Note that since $\mathcal{N}$ is a nuclear countably Hilbert space we have, see e.g.~\cite{HKPS93}:
$$
\mathcal{C}_{\sigma}(\mathcal{N}')=\mathcal{B}_w(\mathcal{N}') =\mathcal{B}_s(\mathcal{N}'),
$$
where $\mathcal{B}_w(\mathcal{N}')$ (resp.~$\mathcal{B}_s(\mathcal{N}')$) is the Borel $\sigma$-algebra generated by the weak (resp. strong) topology.

We define by 
$$\mathcal{P}:= \left\{ p\in L^2(\mathcal{N}'; \mu_{\scriptscriptstyle{H}}) \,\Big|\, p(\omega)= \sum_{n=0}^N \langle \omega^{\otimes n }, f^{\otimes n } \rangle_{\scriptscriptstyle{H}} , \quad f \in \mathcal{N}_{\C}\right\}$$ 
the space of smooth polynomials.

In \cite{BFS16} the authors construct the stochastic quantization of the fractional Edwards measure via a local Dirichlet form. Here we briefly sketch the construction and summarize facts from the differential calculus, needed in this framework. 
\begin{definition}\label{defgrad}
Let $p\in\mathcal{P}$
and $(\eta_k)_{k\in \N} \subset  \mathcal{N}$ a CONS of $\mathcal{H}$. 
Setting 
\begin{align*}
\big(D_{\eta_k} p\big) (\omega) = \lim_{\lambda \to 0} \frac{p(\omega + \lambda \eta_k) - p(\omega)}{\lambda}=\sum_{n=1}^N n \langle \eta_k \otimes \omega^{\otimes n-1} , f^{\otimes n} \rangle_{\scriptscriptstyle{H}},\quad \omega\in\mathcal{N}',
\end{align*}
we define 
$$\nabla p := (D_{\eta_k} p )_{k=1}^{\infty}.$$
\end{definition}
\begin{remark}
Note that this defines $D_{\eta_k}$ and $\nabla$ on a dense subspace of $L^2(\mathcal{N}'; \mu_{\scriptscriptstyle{H}})$.
\end{remark}

For $p\in\mathcal{P}$ we have
\begin{multline*}
\sum_{k=1}^{\infty} (D_{\eta_k} p)^2(\omega) = \sum_{k=1}^{\infty} \sum_{m=1}^N \sum_{n=1}^N m n \langle \omega^{\otimes m-1}\otimes \eta_k , f^{\otimes m} \rangle_{\scriptscriptstyle{H}} \langle \omega^{\otimes n-1}\otimes \eta_k , f^{\otimes n} \rangle_{\scriptscriptstyle{H}} \\
= \sum_{m=1}^N m \sum_{n=1}^N n \big\langle \omega^{\otimes m+n+2}, (f,f)_{\scriptscriptstyle{H}} \,f^{\otimes n+m-2} \big\rangle_{\scriptscriptstyle{H}}.
\end{multline*}
\begin{remark}\label{remadjoint}
Furthermore for $ u \in \mathcal{H}$ the adjoint 
$D^*_u  = \langle \cdot, u \rangle_{\scriptscriptstyle{H}} - D_u $ on a dense subspace, e.g.~polynomials in $L^2(\mathcal{N}'; \mu_{\scriptscriptstyle{H}})$, see e.g.~\cite{Po97}.
\end{remark}

In the following we will just write $L$ for $L(B^{\scriptscriptstyle{H}})$, the \emph{self-intersection local time of} $B^{\scriptscriptstyle{H}}$, $H\in (0,\frac{1}{d})$, where $B^{\scriptscriptstyle{H}}$ is a $d$-dimensional fractional Brownian motion with Hurst parameter $H$. 

\begin{definition}
	By $\nu_{\scriptscriptstyle{g}}:= \exp(-gL)\,\mu_{\scriptscriptstyle{H}}$ we denote the \emph{fractional Edwards measure}. Moreover,  $L^2(\mathcal{N}';\nu_{\scriptscriptstyle{g}})$ is the corresponding space of square integrable functions equipped with the inner product $(\cdot,\cdot)_{\scriptscriptstyle{L^2(\mathcal{N}';\nu_g)}}$.
\end{definition}

\begin{remark}
	Note that since $\exp(-gL)\in L^2(\mathcal{N}';\mu_{\scriptscriptstyle{H}})$ with $dH<1$, we have in this case that $\nu_g$ is absolutely continuous w.r.t.~$\mu_{\scriptscriptstyle{H}}$ for all $g>0$, see e.g.~\cite{Hu2001}.
\end{remark}
\begin{theorem}\label{thm closable} 
The bilinear form
$$ \mathcal{E}_{\nu_g}(u,v) := \mathbb{E}_{\scriptscriptstyle{H}} \big(\exp({-gL}) \nabla u \cdot \nabla v\big), \quad u,v \in \mathcal{P},$$
is a densely defined, closable, symmetric pre-Dirichlet form and gives rise to a local, quasi-regular Dirichlet form $(\mathcal{E}_{\nu_g},D(\mathcal{E}_{\nu_g}))$ in $L^2(\mathcal{N}'; \mu_{\scriptscriptstyle{H}})$.
Here $\mathbb{E}_{\scriptscriptstyle{H}}$ denotes expectation w.r.t.~$\mu_{\scriptscriptstyle{H}}$.
\end{theorem}

\begin{proof}
	See \cite[Theorem 3.1]{BFS16}.
\end{proof}

\begin{remark}\label{rem closable}
	\begin{enumerate}
		\item[(i)]
			In particular, Remark \ref{rem_fullsupp} provides that the bilinear form in Theorem \ref{thm closable} is well-defined. More precisely, the full support of the measure insures that the gradient respects the $\mu_{\scriptscriptstyle{H}}$-classes (hence also the $\nu_g$-classes) determined by $\mathcal{P}$.
		\item[(ii)] 
			Due to Theorem \ref{thm closable} we have that
			\begin{align}\label{rep form}
			\mathcal{E}_{\nu_g}(u,v)=\sum_{k=1}^{\infty} \Big( D_{\eta_k} u,D_{\eta_k}v\Big)_{L^2(\mathcal{N}';\nu_g)}, \quad u,v \in \mathcal{P},
			\end{align}
			is a densely defined, closable, symmetric classical gradient pre-Dirichlet form and gives rise to a local, quasi-regular Dirichlet form $\big(\mathcal{E}_{\nu_g},D(\mathcal{E}_{\nu_g})\big)$ in $L^2(\mathcal{N}';\nu_g)$.
		\item[(iii)]
			Moreover, since $1\in\mathcal{P}$ and $\mathcal{E}_{\nu_g}(1,1)=0$, due to \cite[Theo.~1.6.3]{Fukushima}, the local, quasi-regular Dirichlet form $\big(\mathcal{E}_{\nu_g},D(\mathcal{E}_{\nu_g})\big)$ in $L^2(\mathcal{N}';\nu_g)$ is recurrent.
		\item[(iv)]
			 Due to \cite[Chapter II, Section 3 d)]{MR92} (i) implies that $(\nabla,\mathcal{P})$ is closable in $L^2(\mathcal{N}';\nu_g)$. We denote the closure of $\nabla$ by the same symbol. 
		\item[(v)]
			As in \cite[Corollar 10.8]{HKPS93} we obtain that the closures of $\big(\mathcal{E}_{\nu_g},\mathcal{P}\big)$ and $\big(\mathcal{E}_{\nu_g},\mathcal{F}C_b^\infty\big)$ coincide. Here and below $\psi\in\mathcal{F}C_b^\infty$ is of the form
			\begin{align*}
			\psi(\omega)=f\big(\langle\xi_1,\omega\rangle_{\scriptscriptstyle{H}},\ldots,\langle\xi_n,\omega\rangle_{\scriptscriptstyle{H}}\big),\quad f\in C_b^\infty(\mathbb{R}^n),~\xi_j\in\mathcal{N},~j\in\{1,\ldots,n\},~n\in\mathbb{N},~\omega\in\mathcal{N}'.
			\end{align*}
		%\item[(vi)] Note that (\ref{rep form}) defines a classical Dirichlet form.	
\end{enumerate}
\end{remark}
By Friedrichs representation theorem we have the existence of the self-adjoint generator
$\big({A_{\nu_g}},D({A_{\nu_g}})\big)$ corresponding to $\big(\mathcal{E}_{\nu_g},D(\mathcal{E}_{\nu_g})\big)$.
\begin{proposition}\label{propgen}
	There exists a unique, positive, self-adjoint, linear operator $\big({A_{\nu_g}},D({A_{\nu_g}})\big)$ on $L^2(\mathcal{N}';\nu_g)$ such that
	\begin{align*}
	D({A}_{\nu_g})\subset D(\mathcal{E}_{\nu_g})\quad\text{and}
	\quad\mathcal{E}_{\nu_g}\big(u,v\big)=\Big({A}_{\nu_g} u,v\Big)_{L^2(\mathcal{N}';\nu_g)}
	\quad\text{for all }u\in D({A}_{\nu_g}),~v\in D(\mathcal{E}_{\nu_g}).
	\end{align*}
\end{proposition}
\begin{proof}
	Using Remark \ref{rem closable} this is a direct application of \cite[Coro.~1.3.1]{Fukushima}.
\end{proof}

%\section{Markov uniqueness}

%\section{Representation for the generator}
For functions $u\in\mathcal{P}$ the next result provides a nice representation of the operator from the above proposition.
\begin{proposition}\label{prop_gen_rep}
	For $u\in\mathcal{P}$ the generator $A_{\nu_g}$ in Proposition \ref{propgen} has the form
	\begin{align*}
	A_{\nu_g}u=-\mathcal{L}u:=Nu-g\,\nabla u\cdot\nabla L,
	\end{align*}
	where $Nu:=\sum_{k=1}^\infty D_{\eta_k}^*D_{\eta_k}u$ is the so-called \emph{number operator}.  
\end{proposition}

\begin{proof}
	For $u,v\in\mathcal{P}$ we have
	\begin{multline*}
	\mathcal{E}_{\nu_g}(u,v) = \sum_{k=1}^{\infty} \int_{\mathcal{N}'}D_{\eta_k} u \,D_{\eta_k}v\,e^{-gL} \,d\mu_{\scriptscriptstyle{H}}=\sum_{k=1}^{\infty} \int_{\mathcal{N}'}\Big(D_{\eta_k} u\,e^{-gL} \Big) \,D_{\eta_k}v\,d\mu_{\scriptscriptstyle{H}}\\
	=\sum_{k=1}^{\infty} \int_{\mathcal{N}'}D_{\eta_k}^*\,\Big(\big(D_{\eta_k} u\big)\,e^{-gL} \Big)\,v\,d\mu_{\scriptscriptstyle{H}}=\sum_{k=1}^{\infty}\int_{\mathcal{N}'}\left(\big\langle \cdot,\eta_k\big\rangle_{\scriptscriptstyle{H}}D_{\eta_k}u\,e^{-gL}
	-D_{\eta_k}\big(D_{\eta_k}u\,e^{-gL}\big)\right)\, v\,d\mu_{\scriptscriptstyle{H}}\\
	=\sum_{k=1}^{\infty}\int_{\mathcal{N}'}\left(\big\langle \cdot,\eta_k\big\rangle_{\scriptscriptstyle{H}}D_{\eta_k}u\,e^{-gL}
	-\Big(D_{\eta_k}D_{\eta_k}u\,e^{-gL}-\big(D_{\eta_k}u\big)\,g \big(D_{\eta_k}L\big)\,e^{-gL}\Big)\right)\,v\,d\mu_{\scriptscriptstyle{H}}\\
	=\sum_{k=1}^{\infty}\int_{\mathcal{N}'}\Big(\big\langle \cdot,\eta_k\big\rangle_{\scriptscriptstyle{H}}D_{\eta_k}u
	-D_{\eta_k}D_{\eta_k}u+D_{\eta_k}u\,g D_{\eta_k}L\Big)\,v\,e^{-gL}\,d\mu_{\scriptscriptstyle{H}}\\
	=\sum_{k=1}^{\infty}\int_{\mathcal{N}'}\Big(D_{\eta_k}^*D_{\eta_k}u+g\,\big(D_{\eta_k}u\big)\,\big(D_{\eta_k}L\big)\Big)\,v\,e^{-gL}\,d\mu_{\scriptscriptstyle{H}}.
	\end{multline*}
	Thus by using the so-called \emph{number operator} 
	\begin{align}\label{equnumberop}
	Nu=\sum_{k=1}^\infty D_{\eta_k}^*D_{\eta_k}u
	\end{align}
	for $u\in\mathcal{P}$ we obtain
	\begin{align*}
	\mathcal{E}_{\nu_g}(u,v)=\big(-\mathcal{L}u,v\big)_{L^2(\mathcal{N}';\nu_g)}=\big(Nu-g\,\nabla u\cdot\nabla L,v\big)_{L^2(\mathcal{N}';\nu_g)},\quad u,v\in\mathcal{P}.
	\end{align*}
	Hence for $u\in\mathcal{P}$ the generator $\mathcal{L}$ is given by
	\begin{align*}
	\mathcal{L}u:=-Nu-g\,\nabla u\cdot\nabla L.
	\end{align*}
\end{proof}

\begin{proposition}\label{prop_Markov_unique}
	The generator $\big(A_{\nu_g}, D(A_{\nu_g})\big)$ in Proposition \ref{propgen} is the only Dirichlet operator extending\\ $\big(-\mathcal{L},\mathcal{F}C_b^\infty\big)$, where $-\mathcal{L}u=Nu-g\,\nabla u\cdot\nabla L$ for $u\in \mathcal{F}C_b^\infty$, see Proposition \ref{prop_gen_rep}.	
\end{proposition}

\begin{proof}
	Due to Remark \ref{rem closable}(iv) we have that 	$A_{\nu_g}u=-\mathcal{L}u=Nu-g\,\nabla u\cdot\nabla L$ for all $u\in \mathcal{F}C_b^\infty$. Since $\exp(-gL)$ and $\nabla\exp(-gL)$ are square-integrable w.r.t.~$\mu_{\scriptscriptstyle{H}}$, see \cite{Hu2001}, and $\mu_{\scriptscriptstyle{H}}$ is Gaussian, we obtain the statement by using \cite[Theorem 2.3]{RZ94}. 
\end{proof}

\begin{remark}
	The property provided in Proposition \ref{prop_Markov_unique} is known as \emph{Markov uniqueness}.
\end{remark}

Let $\big(T_t^{{\nu_g}}\big)_{t\ge 0}$ with $T_t^{{\nu_g}}:=\exp\Big(-tA_{\nu_g}\Big)$, $t\ge 0$, denote the corresponding strongly continuous contraction semigroup on $L^2(\mathcal{N}';\nu_g)$, cf.~e.g.~\cite[Chap.~I, Sect. 1,2]{MR92}.
\begin{remark}
	$\big(T_t^{{\nu_g}}\big)_{t\ge 0}$ is recurrent due to Remark \ref{rem closable}(iii). Using \cite[Lemma 1.6.5]{Fukushima} we obtain that $\big(T_t^{{\nu_g}}\big)_{t\ge 0}$ is conservative.
\end{remark}
Abstract Dirichlet form theory provides the following results, compare e.g. \cite{Fukushima} or \cite{MR92}:

\begin{theorem}\label{thm diffusion}
There exists a diffusion process $\mathbf{M} = (\mathbf{\Omega}, \mathcal{F}, (\mathcal{F}_t)_{t\geq0}, (X_t)_{t\geq 0}, (\mathbf{P}_{\omega})_{\omega \in \mathcal{N}'})$ with state space $\mathcal{N}'$ which is properly associated with 
$(\mathcal{E}_{\nu_g}, {D}(\mathcal{E}_{\nu_g}))$, i.e., for all ($\nu_g$-versions of) $f\in L^2(\mathcal{N}',\nu_g)$ and all $t\ge 0$ the function 
\begin{align*}
\mathcal{N}'\ni \omega\mapsto (p_tf)(\omega):=\int_{\mathbf{\Omega}}f(X_t)\,d\mathbf{P}_\omega\in \R,
\end{align*}
is an $\mathcal{E}_{\nu_g}$-quasi-continuous version of $T_t^{{\nu_g}}f$. $\mathbf{M}$ is up to $\nu_g$-equivalence unique (cf.~\cite[Chap.~IV,~Sect. 6]{MR92}). In particular, $\mathbf{M}$ is $\nu_g$-symmetric, i.e.,
\begin{align*}
\int_{\mathcal{N}'}(p_tf)\,g\,d\nu_g=\int_{\mathcal{N}'}f\,(p_tg)\,d\nu_g
\end{align*}
for all bounded measurable functions $f,g:\mathcal{N}'\to\mathbb{R}$,~$t>0$, as well as conservative, i.e., $p_t\mathbbm{1}=\mathbbm{1}$ $\mathcal{E}_{\nu_g}$-q.e.~for all $t\ge 0$ or in other words the diffusion process $\mathbf{M}$ is of infinite life time. Thus $\nu_g$ is an invariant measure for $\mathbf{M}$. 
\end{theorem}

\begin{theorem}\label{theomartingaleprob}
	The diffusion process $\mathbf{M}$ as given in Theorem \ref{thm diffusion} is solving the martingale problem for $\big({A_{\nu_g}},D({A_{\nu_g}})\big)$, i.e., for
	all $u\in D({A_{\nu_g}})$,
	\begin{align*}
	u\big( X_t\big)-u\big(X_0\big)-\int_0^t \big(A_{\nu_g}u\big)\big(X_s\big)\,ds,\quad t\ge 0,
	\end{align*}
	is an $\big(\mathcal{F}_t\big)_{t\ge 0}$-martingale under $\mathbf{P}_{\omega}$ (hence starting in $\omega$) for $\mathcal{E}_{\nu_g}$-quasi all $\omega\in\mathcal{N}'$.
\end{theorem}
\begin{proof}
	The statement follows by Theorem \ref{thm closable} and \cite[Theorem 3.4(i)]{AR}. 
\end{proof}

\section{Irreducibility and extremal Gibbs states}
In this section we provide important consequences of irreducibility of the considered bilinear form. In particular, this means invariance of the associated diffusion process under time translations. For the stochastic quantization of the fractional Edwards measure, this has been shown in \cite{BFS16} and relates to a particular class of measures.
\begin{definition}\label{def_Gibbs}
	For $K\subset\mathcal{H}$, $\mathcal{B}({\mathcal{N}'})$-measurable functions $b_k$, $k\in K$, and $b:=\big(b_k\big)_{k\in K}$ we define $\mathcal{G}^b$ to be the set of all probability measures $\mu$ on $\mathcal{B}({\mathcal{N}'})$ such that for all $k\in K$, $b_k\in L^2(\mathcal{N}';\nu_g)$ and the following integration by parts formula holds:
	\begin{align*}
	\int_{\mathcal{N}'}\frac{\partial u}{\partial k}\, d\mu=-\int_{\mathcal{N}'} u\,b_k\,d\mu\quad\text{for all}\quad u\in\mathcal{P},
	\end{align*}
	where $\frac{\partial u}{\partial k}(\omega):=\frac{d}{ds}u(\omega+sk)\big|_{s=0}$, $\omega\in\mathcal{N}'$. Elements in $\mathcal{G}^b$ are called \emph{Gibbs states associated with $b$}.
\end{definition}

\begin{remark}
Definition \ref{def_Gibbs} coincides with the Definition of a Gibbs state in the sense of \cite{AR} due to Remark \ref{rem closable}(iii) and (iv). 	
\end{remark}

Using Remark \ref{remadjoint} we obtain

\begin{lemma}\label{lemGibbs}
	For $g\ge 0$ the measures $\widetilde{\nu_g}:= \frac{1}{Z}\,\nu_g$, where $Z\in(0,\infty)$ is a normalizing constant, are contained in $\mathcal{G}^b$ with $b=(b_k)_{k\in K}$,  where $b_k:=\langle k,\cdot\rangle_{\scriptscriptstyle{H}}+\langle k,g\nabla L\rangle_{\scriptscriptstyle{H}}$ and $K:=\text{span~}\beta$.  
\end{lemma}

\begin{remark}
Note that $Z$ is the expectation of $\exp(-gL)$ with respect to $\mu_{\scriptscriptstyle{H}}$, which exists due to \cite{Hu2001}.
\end{remark}

In \cite{BFS16} the following theorem is shown.

\begin{theorem}\label{thm irreducible}
	There exists a constant $c_0>0$, such that for all $g<c_0$ the form
	$(\mathcal{E}_{\nu_g}, D(\mathcal{E}_{\nu_g}))$ is irreducible (i.e. $u\in D(\mathcal{E}_{\nu_g}) \text{ with } (\mathcal{E}_{\nu_g}(u,u)=0$ implies $u$ is a constant), equivalently the associated diffusion is invariant under time translations.
\end{theorem}

This has immediate consequences for the diffusion process $\mathbf{M} = (\mathbf{\Omega}, \mathcal{F}, (\mathcal{F}_t)_{t\geq0}, (X_t)_{t\geq 0}, (\mathbf{P}_{\omega})_{\omega \in \mathcal{N}'})$ given in Theorem \ref{thm diffusion}.

\begin{corollary}\label{timeergo}
	There exists a constant $c_0>0$, such that for all $g<c_0$  we have that
	\begin{align*}
	\lim_{t\to\infty}\frac{1}{t}\int_0^t f(X_t)\,ds =\int_{\mathcal{N}'} f\,d\widetilde{\nu_g}
	\end{align*}
	$\mathbf{P}_\omega$-almost surely for quasi every $\omega\in\mathcal{N}'$ and all $f\in L^1(\mathcal{N}';\widetilde{\nu_g})$.
\end{corollary}

\begin{proof}
	Due to \cite[Theo.~4.7.3(iii)]{Fukushima} it is sufficient to show that $\big(\mathcal{E}_{\nu_g},D(\mathcal{E}_{\nu_g})\big) $ is irreducible recurrent. Recurrence follows by \cite[Theo.~1.6.3]{Fukushima}, since $1\in\mathcal{P}$ and $\mathcal{E}_{\nu_g}(u,u)=0$. Irreducibility is provided by Theorem \ref{thm irreducible}.
\end{proof}

\begin{remark}
	The property obtained in Corollary \ref{timeergo} is know as \emph{time ergodicity} of the process $\big(X_t\big)_{t\ge 0}$.
\end{remark}

Next we introduce extremal Gibbs states.

\begin{definition}
	A measure $\mu\in \mathcal{G}^b$ is called \emph{extremal} if it can not be written as convex combination of elements in the set $\mathcal{G}^b$. This we denote by
	$\mu\in\mathcal{G}_{\scriptscriptstyle{\text{ext}}}^b$.
\end{definition}

\begin{definition}
	Let $\nu$ be a measure on $(\mathcal{N}',\mathcal{B}(\mathcal{N}')$. A $\mathcal{B}(\mathcal{N}')$-measurable (real valued) function $f$ is called \emph{$K$-shift invariant} if $f(\omega+tk)=f(\omega)$ for $\nu$-a.e.~$\omega\in\mathcal{N}'$, for all $t\in\mathbb{R}$ and all $k\in K$. $g\in L^2(\mathcal{N}';\nu)$ is called \emph{$K$-shift invariant} if there exists a $\mathcal{B}(\mathcal{N}')$-measurable representative which is $K$-shift invariant.
\end{definition} 

\begin{corollary}\label{shiftergo}
	In situation of Theorem \ref{thm irreducible} we have that 
	\begin{enumerate}
		\item [(i)]
		$\widetilde{\nu_g}\in\mathcal{G}_{\scriptscriptstyle{\text{ext}}}^b$, i.e., $\widetilde{\nu_g}$ is an extremal Gibbs state.
		\item[(ii)]
		$\widetilde{\nu_g}$ is $K$-ergodic, i.e., every $K$-invariant $\mathcal{B}(H)$-measurable function is $\widetilde{\nu_g}$-a.e.~constant.	
		\item[(iii)]
		the semigroup $\big(T_t^{{\nu_g}}\big)_{t\ge 0}$ is irreducible.
	\end{enumerate}
\end{corollary}
\proof{Apply \cite[Theorems 1.2,~3.7 and Propositopn 2.3]{AKR}}.

\begin{remark}
	\begin{enumerate}
	\item[(i)] Using a standard approximation argument via the CONS of $\mathcal{H}$ we even obtain the results of Corollary \ref{shiftergo} for $K=\mathcal{H}$. 
	\item[(ii)]
	The property obtained in Corollary \ref{shiftergo}(ii) is known as \emph{shift ergodicity} of the process $\big(X_t\big)_{t\ge 0}$.
	\end{enumerate}
\end{remark}

\section{Characterization of the underlying stochastic differential equation}
Abstract Dirichlet form theory provides the following statement, see e.g.~\cite[Theo.~4.3]{AR}.
\begin{theorem}[Fukushima decomposition]\label{theofukdecomp}
	Let $u\in D(\mathcal{E}_{\nu_g})$ and $\widetilde{u}$ a quasi-continuous $\nu_g$-version of $u$. Then the additive functional $\Big(\widetilde{u}\big(X_t\big)-\widetilde{u}\big( X_0\big)\Big)_{t\ge 0}$ of $\mathbf{M}$ can be uniquely represented as
\begin{align}\label{equFukdec}
\widetilde{u}\big(X_t\big)-\widetilde{u}\big( X_0\big)=M_t^{[u]}+N_t^{[u]},\quad t\ge 0,
\end{align}
where $M^{[u]}:=\Big(M_t^{[u]}\Big)_{t\ge 0}$ is a MAF (martingale additive functional ) of $\mathbf{M}$ of finite energy and $N^{[u]}:=\Big(N_t^{[u]}\Big)_{t\ge 0}$ is a CAF (continuous additive functional ) of $\mathbf{M}$ of zero energy. Recall that $\mathbf{M}$ is provided in Theorem \ref{thm diffusion}.
\end{theorem}
 For $k\in K=\text{span}\,\beta$ and $\omega\in\mathcal{N}'$ we define $u_k(\omega):=\langle k,\omega\rangle_{\scriptscriptstyle{{H}}}$.
 \begin{remark}
 	In the situation of Theorem \ref{theofukdecomp}, we have for $k\in K$ that $u_k\in D(A_{\nu_g})$ and $A_{\nu_g}u_k=b_k$. This immediately implies that for $k\in K$, $N^{[u_k]}=\Big(N_t^{[u_k]}\Big)_{t\ge 0}$ in (\ref{equFukdec}) reads
 	\begin{align*}
 	\Big(N_t^{[u_k]}\Big)_{t\ge 0}=\left(\int_0^t b_k\big(X_s\big)\big)\,ds\right)_{t\ge 0}.
 	\end{align*}
 	by Theorem \ref{theomartingaleprob}. Moreover, $\Big(M_t^{[u_k]},\mathcal{F}_t, \mathbf{P}_{\omega}\Big)_{t\ge 0}$ is a martingale which is hence also continuous with $M_0^{[u_k]}=0$. Since
 	\begin{align*}
 	2 \mathcal{E}_{\nu_g}(u_k\cdot \mathbbm{1}, u_k)-\mathcal{E}_{\nu_g}(u_k^2,\mathbbm{1})=2\,\Vert k\Vert^2_{\scriptscriptstyle{{H}}},
 	\end{align*}
the quadratic variation $\left\langle M^{[u_k]}\right\rangle_t$ for $t\ge 0$ is given by
\begin{align*}
\left\langle M^{[u_k]}\right\rangle_t=2t\Vert k\Vert^2_{\scriptscriptstyle{{H}}}.
\end{align*}
If $\Vert k\Vert_{\scriptscriptstyle{{H}}}=1$ it follows by P.~Levy's characterization of Brownian motion and its scaling properties that $\Big(M_t^{[u_k]}\Big)_{t \ge 0}$ is an $(\mathcal{F}_t)_{t\ge 0}$-Brownian motion $\big(W^k_t\big)_{t\ge 0}$ scaled by $\sqrt{2}$ starting at zero under each $\mathbf{P}_{\omega}$ for $\omega\in \mathcal{N}'\setminus S_k$, where $S_k\subset{\mathcal {N}'}$ is a set with capacity zero. This Brownian motion is associated to a gradient bilinear form w.r.t.~the reference measure $\mu_{\scriptscriptstyle{H}}$. The associated generator is given by the number operator $N$, see (\ref{equnumberop}). Due to \cite{Kuo} such a process defines an Ornstein-Uhlenbeck process. In this sense, the appearing Brownian motion $\big(W^k_t\big)_{t\ge 0}$ has an intrinsic linear drift.
\end{remark}
\begin{corollary}
	Let $k\in K$. Then the decomposition (\ref{equFukdec}) reads
	\begin{align}\label{decomp}
	u_k\big(X_t\big)-u_k\big(X_0\big)=\sqrt{2}\, W^k_t+\int_0^t b_k\big(X_t\big)\,ds,\quad t\ge 0,
	\end{align}
	where for all $\omega\in \mathcal{N}'\setminus S_k$ for some $S_k\subset \mathcal{N}'$ with capacity zero, the continuous martingale $\Big( M_t^{[u_k]},\mathcal{F}_t, P_z\Big)_{t\ge 0}$ is a Brownian motion scaled by $\sqrt{2}$ and $b_k=\langle k,\cdot\rangle_{\scriptscriptstyle{H}}+\langle k,g\nabla L\rangle_{\scriptscriptstyle{H}}$ is as in Lemma \ref{lemGibbs}.  
\end{corollary}

\begin{lemma}
	Let $k,k'\in K$. Then
	\begin{align*}
	\left\langle W^{k},W^{k'}\right\rangle_t=t\,(k,k')_{{H}},
	\end{align*}
	under $P_{\omega}$ (as given in Theorem \ref{thm diffusion}) for quasi-every $\omega\in \mathcal{N}'$, where $\big(W^k_t\big)_{t\ge 0}$ is as in (\ref{decomp}).
\end{lemma}
\begin{proof}
See \cite[Lemm.~5.4]{AR}
\end{proof}
This immediately implies the next proposition, see \cite[Prop.~5.5]{AR}. 
\begin{proposition}
For $d\in\mathbb{N}$ let $\eta_1,\ldots,\eta_d\in\beta$. Then $\overline{W}_t:=\big(W^{k_1}_t,\ldots W^{k_d}_t\big)$, $t\ge 0$, is a $d$-dimensional $(\mathcal{F}_t)_{t\ge 0}$-Brownian motion starting at zero under each $P_{\omega}$ for $\omega\in \mathcal{N}'\setminus S_k$, where $S_k\subset{\mathcal {N}'}$ is a set with capacity zero.
\end{proposition}
Moreover, we have by \cite[Theo.~5.7]{AR} the following theorem.
\begin{theorem}\label{solsde}
	For quasi every $\omega\in\mathcal{N}'$, $\Big(\big\{\langle k,\omega\rangle_{\scriptscriptstyle{H}}\,|\,k\in\beta\big\},\mathcal{F}_t,P_{\omega}\Big)_{t\ge 0}$ solves the following system  of stochastic differential equations
	\begin{align}\label{sde1}
	\begin{array}{c}
	dY_t=\sqrt{2}\,dW^k_t+b_k\big((Y_t^k)_k\in\beta\big)\,dt\\
	Y_0^k=\langle k,z\rangle_{\scriptscriptstyle{H}}
	\end{array}, k\in\beta,
	\end{align}
	where $\Big\{\big(W_t^k\big)_{t\ge 0}\,\Big|\,k\in\beta\Big\}$ is a collection of independent one dimensional $(\mathcal{F}_t)_{t\ge 0}$-Brownian motions starting at zero, where we identify $z\in\mathcal{N}'$ with $\big(\langle k,z\rangle_{\scriptscriptstyle{H}}\big)_{k\in\beta}$, and $b_k=\langle k,\cdot\rangle_{\scriptscriptstyle{H}}+\langle k,g\nabla L\rangle_{\scriptscriptstyle{H}}$ is as in Lemma \ref{lemGibbs}.   
\end{theorem}
\begin{remark}
	Theorem \ref{solsde} just says that using the corresponding Dirichlet form we have constructed a weak solution of (\ref{sde1}), which is unique by Proposition \ref{prop_Markov_unique} and Theorem \ref{theomartingaleprob}.
\end{remark}

Applying \cite[Theo.~6.10]{AR} we obtain the following main result.
\begin{theorem}\label{theoweaksol}
	For $\mathbf{M} = (\mathbf{\Omega}, \mathcal{F}, (\mathcal{F}_t)_{t\geq0}, (X_t)_{t\geq 0}, (\mathbf{P}_{\omega})_{\omega \in \mathcal{N}'})$, see Theorem \ref{thm diffusion}, there exists a map $W:\mathcal{N}'\to C([0,\infty),\mathcal{N}')$ such that for quasi all $\omega\in\mathcal{N}'$ under $\mathbf{P}_{\omega}$, $W=(W_t)_{t\ge 0}$ is an $(\mathcal{F}_t)_{t\ge 0}$-Brownian motion on $\mathcal{N}'$ starting in zero with covariance $(\cdot,\cdot)_{\scriptscriptstyle{H}}$ such that for quasi every $\omega\in\mathcal{N}'$ we have the unique representation
	\begin{align*}
	X_t = \omega + \sqrt{2}\,W_t+\int_0^t b(X_s)\,ds,\quad t\ge 0,\quad \mathbf{P}_{\omega}\text{-a.s.}
	\end{align*}
with $b=(b_k)_{k\in \beta}$,  where $b_k=\langle k,\cdot\rangle_{\scriptscriptstyle{H}}+\langle k,g\nabla L\rangle_{\scriptscriptstyle{H}}$.
\end{theorem}

\begin{remark}
	Theorem \ref{theoweaksol} just says that the diffusion process $\mathbf{M}$ in Theorem \ref{thm diffusion} provides a unique weak solution to the stochastic differential equation
	\begin{align}\label{sde}
	dX_t = \sqrt{2}\,dW_t + b(X_t)\,dt,\quad X_0=\omega,
	\end{align}
	with $b=(b_k)_{k\in \beta}$,  where $b_k=\langle k,\cdot\rangle_{\scriptscriptstyle{H}}+\langle k,g\nabla L\rangle_{\scriptscriptstyle{H}}$.
\end{remark}

\section{Continuity of polymer paths}
The stochastic differential equation we obtained by the Fukushima decomposition gives in every compute time point $t$ a noise for a path, which for large compute times, due to the long time behavior of the process is distributed according to the law of a weakly self-avoiding fractional Brownian motion starting in $0$. However, since we constructed the process by stochastic quantization on the level of the noise, the question is, if we have polymer paths which are continuous. \\
Here we want to emphasize that we are not talking about the continuous paths properties of the process which is given by the SDE (\ref{sde}), but more about the continuity property of the process at a compute time point $t$ integrated out on the function time interval $[0,\tau]$. The property thus is also dependent on the construction of the space. 

\begin{proposition}\label{propContpathfbm}
	For all $H\in (0,1)$ the stochastic process $\langle \cdot, \mathbbm{1}_{[0,\tau)}\rangle_{\scriptscriptstyle{H}}$ has $\mu_{\scriptscriptstyle{H}}$-almost surely continuous paths. In other words: For almost every $\omega \in \mathcal{N}' $ the mapping $\tau\mapsto \langle \omega, \mathbbm{1}_{[0,\tau)}\rangle$ is continuous.  
\end{proposition}
 A proof via Kolmogorov-Chentsov can be found e.g.~in \cite{Mishura, oks}.\\
\begin{remark}
	Note that the SDE (\ref{sde1}) provides in the language of polymer physics the dynamics of the bonds, i.e.~stochastically spoken the noise of the process. An integration leads then to polymer paths. Here this is done by a dual pairing with the indicator functions, which exists in the sense of an $L^2(\mathcal{N}';{\mu_{\scriptscriptstyle{H}}})$-limit, compare \cite{HKPS93}.
\end{remark}

\begin{proposition}
	For an initial state $\omega \in \mathcal{N}'$ with $\tau \mapsto \langle \omega, \mathbbm{1}_{[0,\tau)}\rangle$ is continuous the Markov process $\mathbf{M}$ given by Theorem \ref{theoweaksol} fulfills for every compute time point $t$ that 
	$$t \mapsto \langle {X}_{t}(\omega), \mathbbm{1}_{[0,\tau)}\rangle $$ is $\mu_{\scriptscriptstyle{H}}$-almost surely continuous. 
\end{proposition}

\begin{proof}
	By Proposition \ref{propContpathfbm} we have that $\langle \cdot, \mathbbm{1}_{[0,t)}\rangle_{\scriptscriptstyle{H}}$ has $\mu_{\scriptscriptstyle{H}}$-almost surely continuous paths. Denote by $R_t$ the set of points in $\mathcal{N}'$ which can be reached by $X$ at time point $t$ with $X_0=\omega$. Therefore it suffices to show, that $\mu_{\scriptscriptstyle{H}}(R_t)>0$ for all $t$. Since an Ornstein-Uhlenbeck process puts mass on such a non-null set for every compute time point $t$, see e.g~\cite{Kuo} the assertion is shown. 
\end{proof}

\begin{remark}
	The above proposition makes sure that if one starts with a continuous polymer configuration during the evolution the process stays almost surely continuous. 
\end{remark}

\noindent{\bf{Acknowledgement:} }We truly thank L.~Streit for helpful discussions. Financial support by the mathematics department of the University of Kaiserslautern for research visits at Lisbon are gratefully acknowledged.

\end{document}